\documentclass[12pt, reqno]{amsart}
\usepackage{amsmath,times,epsfig,amssymb,amsbsy,amscd,amsfonts,amstext,color,bm}
\usepackage[arrow,matrix]{xy}

\usepackage{bm}

\usepackage{tikz}

\numberwithin{equation}{section}

\theoremstyle{plain}
\newtheorem{theorem}{Theorem}[section]

\newtheorem{lemma}[theorem]{Lemma}

\newtheorem{proposition}[theorem]{Proposition}
\newtheorem{conjecture}[theorem]{Conjecture}

\theoremstyle{definition}
\newtheorem{definition}[theorem]{Definition}
\newtheorem{example}[theorem]{Example}

\theoremstyle{remark}
\newtheorem{remark}[theorem]{Remark}

\newcommand{\A}{\mathcal{A}}

\newcommand{\Z}{\mathbb{Z}}

\newcommand{\R}{\mathbb{R}}
\newcommand{\C}{\mathbb{C}}

\newcommand{\Cat}{\operatorname{Cat}}

\newcommand{\Der}{\operatorname{Der}}

\newcommand{\Ker}{\operatorname{Ker}}

\newcommand{\Spec}{\operatorname{Spec}}

\begin{document}


\title{Integral expressions for derivations of multiarrangements}

\begin{abstract}
The construction of an explicit basis for a free multiarrangement 
is not easy in general. Inspired by the integral expressions for 
quasi-invariants of quantum Calogero-Moser systems, 
we present integral expressions for specific bases of 
certain multiarrangements. Our construction covers the cases of 
three lines in dimension $2$ (previously examined by Wakamiko) and 
free multiarrangements associated with 
complex reflection groups (Hoge, Mano, R\"ohrle, Stump). 
Furthermore, we propose a conjectural basis for the module of 
logarithmic vector fields of 
the extended Catalan arrangement of type $B_2$. 
\end{abstract}

\author{Misha Feigin}
\address{Misha Feigin, 
University of Glasgow}
\email{misha.feigin@glasgow.ac.uk}

\author{Zixuan Wang}
\address{Zixuan Wang, 
Hokkaido University
}
\email{zixuan.wang.o1@elms.hokudai.ac.jp}

\author{Masahiko Yoshinaga}
\address{Masahiko Yoshinaga, 
Osaka University}
\email{yoshinaga@math.sci.osaka-u.ac.jp}

\thanks{
The authors thank Takuro Abe and Shuhei Tsujie 
for helpful discussions. 
They also thank Theo Douvropoulos and Christian Stump for helpful discussions related to Proposition \ref{multibraid}. 
M. F. was partially supported by the EPSRC grant EP/W013053/1. 
Z. W. was supported by JST SPRING, Grant Number JPMJSP2119. 
M. Y. was partially supported by 
JSPS KAKENHI Grant Numbers JP18H01115, JP15KK0144, JP23H00081}



\date{\today}

\maketitle


\section{Introduction}
\label{sec:intro}

Let $V$ be an $\ell$-dimensional linear space over $\C$. 
Let $\{x_1, \dots, x_\ell\}$ be a basis of $V^*$ and let 
$S=S(V^*)=\C[x_1, \dots, x_\ell]$ be the polynomial ring. 
Define $\Der_S=\bigoplus_{i=1}^\ell S\partial_i$ as the module of 
polynomial vector fields ($\C$-linear derivations of $S$), 
where $\partial_i=\frac{\partial}{\partial x_i}$. 

Let $\A=\{H_1, \dots, H_n\}$ be a central arrangement of 
hyperplanes. The pair $(\A, m)$ consisting of $\A$ and a map 
$m:\A\longrightarrow\Z_{\geq 0}$ is called a multiarrangement. 
For each $H\in\A$, choose a linear form $\alpha_H\in V^*$ such that 
$H=\Ker(\alpha_H)$. The polynomial 
\[
Q(\A, m)=\prod_{H\in\A}\alpha_H^{m(H)}
\]
is the defining equation of the multiarrangement $(\A, m)$. 

For a multiarrangement $(\A, m)$, 
define the module of multiderivations $D(\A, m)$ by 
\begin{equation}
D(\A, m):=\{\delta\in\Der_S\mid \delta\alpha_H\in(\alpha_H^{m(H)}), 
\mbox{ for any }H\in\A\}. 
\end{equation}
The multiarrangement $(\A, m)$ is said to be free with 
exponents $(e_1, \dots, e_\ell)$ if $D(\A, m)$ is a free $S$-module 
with a basis $\{\delta_1, \dots, \delta_\ell\}\subset D(\A, m)$ of the 
form $\delta_i=\sum_{j=1}^\ell f_{ij}\partial_j$, 
where each non-zero $f_{ij}$ 
is a homogeneous polynomial of $\deg f_{ij}=e_i$. 
The module $D(\A, m)$ was first introduced by K. Saito 
in \cite{sai-log} for $m\equiv 1$ and by Ziegler \cite{zie} 
for general $m$. 
$D(\A, 1)$ is simply denoted by $D(\A)$. 
We also define $\deg \delta_i = e_i$.

We will frequently use the following criterion for the freeness 
of $D(\A, m)$. 
\begin{proposition}
\label{prop:szc}
(Saito-Ziegler criterion for freeness) 
Let $(\A, m)$ be a multiarrangement. 
Suppose that there exist homogeneous vector fields 
$\delta_1, \dots, \delta_\ell\in D(\A, m)$ such that 
\begin{itemize}
\item 
$\delta_1, \dots, \delta_\ell$ are linearly independent over $S$, and 
\item 
$\sum_{i=1}^\ell\deg\delta_i=\sum_{H\in\A} m(H)$. 
\end{itemize}
Then $D(\A, m)$ is a free $S$-module with a basis 
$\{\delta_1, \dots, \delta_\ell\}$. 
\end{proposition}

The module $D(\A)$ plays a crucial role in the study 
of hyperplane arrangements \cite{ot}. 
The algebraic structures of these modules are thought to reflect 
the combinatorial structures of $\A$. 
For example, when $D(\A)$ is a free module, Tearo's factorization 
theorem \cite{ter-fact} asserts that the characteristic polynomial 
of $\A$ factors into the product of linear terms over $\Z$. 
The module $D(\A, m)$ of multiderivations is also important 
for characterizing the freeness of $D(\A)$ 
\cite{yos-char, yos-survey}. 

Multiderivations for Coxeter arrangements have various 
relations with other research topics, e.g., flat structures 
\cite{sai-prim, sai-lin, sai-unif}, Frobenius manifolds \cite{dub}. 
In recent years, the connection between such multiderivations and 
quantum integrable systems has become apparent \cite{aefy}, 
which is linked 
to significant advancements in \cite{arsy, hmrs}. 
Indeed, it has been observed that (the invariant part of) the 
module $D(\A, m)$ is identified with a space of quasi-invariants 
for the quantum generalized Calogero-Moser system \cite{cha-ves, fei-ves}. 
The relationship between these two fields is expected to 
enhance our understanding of both research topics. 

Despite numerous studies, it is still difficult to construct 
an explicit basis for the module $D(\A, m)$ of 
multiderivations, even in simple cases, e.g., 
$2$-dimensional arrangements \cite{wak-exp, wak-yuz}. 

The purpose of this paper is to investigate integral expressions 
for the basis of $D(\A, m)$, inspired by the integral expressions 
of quasi-invariants \cite{bm, cha-ves, fei-ves, fv}. 
The paper is organized as follows. In \S \ref{sec:3lines}, 
we examine the basis for the multiarrangement on three lines, 
specifically that of the form $x_1^p x_2^q (x_1-x_2)^r$. 
The freeness of $D(\A, m)$ and the exponents are already known 
\cite{zie, wak-yuz}, as well as an existing basis \cite{wak-bas}. 
However, in this work, we present a novel and 
convenient expression for the basis using integrals. 
Additionally, we discuss a higher dimensional analogue. 

In \S \ref{sec:main}, we discuss multiarrangements 
associated with complex reflection groups. In dimension $\geq 3$, 
the freeness of $D(\A, m)$ depends on the multiplicity $m$. 
In \S \ref{subsec:freemultiarr}, 
we recall two families of free multiarrangements 
$(\A, \mu)$ and $(\A, \mu')$, discovered by 
Hoge, Mano, R\"ohrle, and Stump \cite{hmrs}. 
Subsequently, in \S \ref{subsec:mu} and \S \ref{subsec:muprime}, 
we provide explicit integral expressions for the basis. 
In \S \ref{subsec:special}, we consider special cases, 
including constant multiplicities for type $B_\ell$ and 
$D_\ell$ arrangements. 

In \S \ref{sec:prim}, we discuss the relationship between 
integral expressions and K. Saito's primitive derivation. 
Saito first introduced the primitive derivation and the Hodge 
filtration on the module of logarithmic derivations in the 
context of singularity theory \cite{sai-prim}. 
Later, Saito described these 
structures in terms of invariant theory for Coxeter 
arrangements \cite{sai-lin, sai-unif}. Recently, these 
structures have been generalized to well-generated 
complex reflection groups \cite{hmrs} based on 
the study of the Okubo system of linear differential equations 
\cite{kms}. We observe that our integral expressions are 
consistent with the primitive derivation and the Hodge filtration. 

In \S \ref{sec:conj}, we propose a conjectural basis for the 
extended Catalan arrangements of type $B_2$.

\section{Integral expressions for certain multiarrangements}
\label{sec:3lines}

\subsection{Three lines}

In this section, we consider the multiarrangement $(\A, m)$ 
defined by $x_1^p x_2^q (x_1-x_2)^r$ in $\R^2$ 
(Figure \ref{fig:three}). 
\begin{figure}[htbp]
\centering
\begin{tikzpicture}[scale=0.8]


\draw[thick, black] (2,0) -- +(0,4) node[above] {$p$} ; 
\draw[thick, black] (0,2)  -- +(4,0) node[right] {$q$}; 
\draw[thick, black] (0.5,0.5)  -- +(3,3) node[right] {$r$}; 

\end{tikzpicture}
\caption{The multiarrangement $x_1^p x_2^q (x_1-x_2)^r$} 
\label{fig:three}
\end{figure}
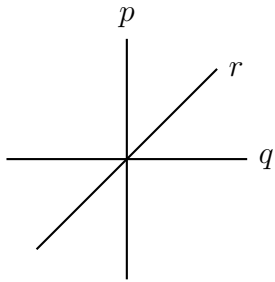
Without loss of generality, we may assume 
\begin{equation}
\max\{p, q\}\leq r. 
\end{equation}
If $r\geq p+q-1$, then 
it is easily seen that the vector fields 
\begin{equation}
\begin{split}
&x_1^p x_2^q(\partial_1+\partial_2), \mbox{ and }\\
&\left(\sum_{i=p}^r\binom{r}{i}(-x_1)^i x_2^{r-i}\right)\partial_1-
\left(\sum_{i=0}^{p-1}\binom{r}{i}(-x_1)^i x_2^{r-i}\right)\partial_2
\end{split}
\end{equation}
form a basis of $D(\A, m)$. 
When $r< p+q-1$, the construction of the basis is more involved. 
Wakamiko \cite[Theorem 1.5]{wak-exp} constructed a basis using 
the generalized binomial coefficients. 
Here, we construct a basis with an integral expression. 

\begin{definition}
Let $a, b, c\in\Z_{\geq 0}$. 
Define $\theta_{a, b, c}\in\Der_S$ to be 
\[
\theta_{a, b, c}=
\left(
\int^{x_1}t^c(t-x_1)^b(t-x_2)^adt
\right)\partial_1
+
\left(
\int^{x_2}t^c(t-x_1)^b(t-x_2)^adt
\right)\partial_2, 
\]
where $\int^{x_i}dt$ is the linear operator defined by 
$\int^{x_i}t^kdt=\frac{x_i^{k+1}}{k+1}$ for $k\neq -1$. 
When $k\geq 0$, it is equivalent to 
the definite integral $\int_0^{x_i}dt$. Later we will also need the 
same operator for $k<-1$. In this case, 
$\int^{x_i}dt$ is equivalent to $\int_\infty^{x_i}dt$.  
\end{definition}
The following are the basic properties of $\theta_{a, b, c}$. 
\begin{proposition}
\label{prop:contain}
The degree of $\theta_{a, b, c}$ is 
$\deg\theta_{a, b, c}=a+b+c+1$ and $\theta_{a, b, c}$ is contained in 
$D(\A, m)$ for $Q(\A, m)=x_1^{b+c+1}x_2^{a+c+1}(x_1-x_2)^{a+b+1}$. 
\end{proposition}
\begin{proof}
The claim 
$\deg\theta_{a, b, c}=a+b+c+1$ is clear. 
Let us compute the order of vanishing of $\theta_{a, b, c}(x_1-x_2)$ at $x_1=x_2$. 
Let $t=x_1+s$. Then 
\[
\begin{split}
\theta_{a, b, c}(x_2-x_1)&=\int_{x_1}^{x_2}t^c(t-x_1)^b(t-x_2)^adt\\
&=
\int_{0}^{x_2-x_1}(x_1+s)^c s^{b} (s-(x_2-x_1))^ads
\end{split}
\]
which is divisible by $(x_1-x_2)^{a+b+1}$. 
Other assertions are proved similarly. 
\end{proof}

\begin{theorem}
Let $(\A, m)$ be the multiarrangement defined by 
$x_1^px_2^q(x_1-x_2)^r$ with $p, q>0$ and $r\leq p+q-1$. 
\begin{itemize}
\item[(1)] 
Suppose $p+q+r$ is odd. Let 
\[
(a, b, c)=
\left(
\frac{-p+q+r-1}{2}, 
\frac{p-q+r-1}{2}, 
\frac{p+q-r-1}{2}
\right).
\]
Then $\theta_{a, b, c}$ and $\theta_{a, b, c+1}$ form a basis of 
$D(\A, m)$. 
\item[(2)] 
Suppose $p+q+r$ is even. Let 
\[
(a, b, c)=
\left(
\frac{-p+q+r}{2}, 
\frac{p-q+r-2}{2}, 
\frac{p+q-r-2}{2}
\right), 
\]
and 
\[
(a', b', c')=
\left(
\frac{-p+q+r-2}{2}, 
\frac{p-q+r}{2}, 
\frac{p+q-r-2}{2}
\right). 
\]
Then $x_1\theta_{a, b, c}$ and $x_2\theta_{a', b', c'}$ 
form a basis of $D(\A, m)$. 
\end{itemize}
\end{theorem}
\begin{proof}
(1) It is clear that $\theta_{a, b, c}$ and $\theta_{a, b, c+1}$ are 
contained in $D(\A, m)$ by Proposition \ref{prop:contain}. 
Since $\deg\theta_{a, b, c}=\frac{p+q+r-1}{2}$ and 
$\deg\theta_{a, b, c+1}=\frac{p+q+r+1}{2}$ are the exponents 
of $(\A, m)$ (\cite{wak-exp}), it suffices to show that 
these two vector fields are linearly independent over $S$, 
or equivalently, the Saito determinant 
\[
\det
\begin{pmatrix}
\theta_{a, b, c}x_1&
\theta_{a, b, c}x_2\\
\theta_{a, b, c+1}x_1&
\theta_{a, b, c+1}x_2
\end{pmatrix}
\]
is nonzero. The highest degree terms 
with respect to the variable $x_1$ in each entry are 
\[
\begin{pmatrix}
x_1^{a+b+c+1}&
x_1^bx_2^{a+c+1}\\
x_1^{a+b+c+2}&
x_1^bx_2^{a+c+2}
\end{pmatrix}, 
\]
up to nonzero constant factors. 
Then, it is clear that the highest degree 
term in the determinant is 
a non-zero multiple of 
$x_1^{a+2b+c+2}x_2^{a+c+1}$.

(2) 
It is clear that $x_1 \theta_{a, b, c}$ and $x_2 \theta_{a', b', c'}$ are 
contained in $D(\A, m)$ by Proposition \ref{prop:contain}.  
Since $\deg x_1\theta_{a, b, c}=\deg x_2\theta_{a', b', c'}=
\frac{p+q+r}{2}$ are the exponents 
of $(\A, m)$ (\cite{wak-exp}), it suffices to show that 
these two vector fields are linearly independent over $S$, 
equivalently, the Saito determinant 
\[
\det
\begin{pmatrix}
x_1\theta_{a, b, c}x_1&
x_1\theta_{a, b, c}x_2\\
x_2\theta_{a', b', c'}x_1&
x_2\theta_{a', b', c'}x_2
\end{pmatrix}
\]
is nonzero. Note that $(a', b', c')=(a-1, b+1, c)$. 
The highest degree terms 
with respect to the variable $x_1$ in each entry are 
\[
\begin{pmatrix}
x_1^{a+b+c+2}&
x_1^{b+1}x_2^{a+c+1}\\
x_1^{a+b+c+1}x_2&
x_1^{b+1}x_2^{a+c+1}
\end{pmatrix}, 
\]
up to nonzero constant factor. Then, the highest degree 
term in the determinant is 
a non-zero multiple of 
$x_1^{a+2b+c+3}x_2^{a+c+1}$.
\end{proof}

\begin{example}
The following are generic examples. 

(1) 
Let $(\A, m)$ be the multiarrangement defined 
by $x_1^{101}x_2^{115}(x_1-x_2)^{157}$. 
Then the following vector fields 
form a basis of $D(\A, m)$ with exponents $(186, 187)$: 
\[
{\small
\begin{split}
&
\left(\int^{x_1}t^{29}(t-x_1)^{71}(t-x_2)^{85}dt\right)\partial_1+
\left(\int^{x_2}t^{29}(t-x_1)^{71}(t-x_2)^{85}dt\right)\partial_2,\\
&
\left(\int^{x_1}t^{30}(t-x_1)^{71}(t-x_2)^{85}dt\right)\partial_1+
\left(\int^{x_2}t^{30}(t-x_1)^{71}(t-x_2)^{85}dt\right)\partial_2. 
\end{split}
}
\]

(2) 
Let $(\A, m)$ be the multiarrangement defined by 
$x_1^{100}x_2^{115}(x_1-x_2)^{157}$. 
Then the following vector fields 
form a basis of $D(\A, m)$ with exponents $(186, 186)$: 
\[
{\footnotesize 
\begin{split}
&
x_1\left\{
\left(\int^{x_1}t^{28}(t-x_1)^{70}(t-x_2)^{86}dt\right)\partial_1+
\left(\int^{x_2}t^{28}(t-x_1)^{70}(t-x_2)^{86}dt\right)\partial_2
\right\},\\
&
x_2\left\{
\left(\int^{x_1}t^{28}(t-x_1)^{71}(t-x_2)^{85}dt\right)\partial_1+
\left(\int^{x_2}t^{28}(t-x_1)^{71}(t-x_2)^{85}dt\right)\partial_2
\right\}.\\
\end{split}
}
\]
\end{example}

\subsection{A higher-dimensional analogue}

Let $\ell\geq 2$. 
Let $a_1, a_2, \dots, a_\ell, b\in\Z_{\geq 0}$. 
We define the vector field 
$\theta_{a_1, \dots, a_\ell, b}$ as 
\begin{equation}
\theta_{a_1, \dots, a_\ell, b}:=
\sum_{i=1}^\ell
\left(
\int^{x_i}t^b(t-x_1)^{a_1}\dots(t-x_\ell)^{a_\ell}dt
\right)\partial_i. 
\end{equation}
Note that $\deg \theta_{a_1,\dots, a_\ell, b}=\sum_{i=1}^\ell a_i+b+1$. 
By a similar computation to the proof of 
Proposition \ref{prop:contain}, we can easily verify the following. 

\begin{proposition}
\label{prop:higherdimmulti}
Let $(\A, m)$ be the multiarrangement defined by 
\begin{equation}
\label{mbraidl}
Q(\A, m)=\prod_{i=1}^\ell x_i^{a_i+b+1}\cdot
\prod_{1\leq i<j\leq \ell}(x_i-x_j)^{a_i+a_j+1}. 
\end{equation}
Then the vector field $\theta_{a_1, \dots, a_\ell, b}$ is contained in 
$D(\A, m)$. 
\end{proposition}
The vector fields $\theta_{a_1, \dots, a_\ell, b}$ 
are useful 
for constructing elements in the logarithmic vector fields. 
The following 
statement gives one of the basic properties. 

\begin{lemma}
\label{lem:indep}
Let $b_1, \dots, b_\ell\in\Z_{\geq 0}$ with 
$b_1<b_2<\dots<b_\ell$. 
Then the $\ell$ vector fields 
$\theta_{a_1, \dots, a_\ell, b_1}, 
\theta_{a_1, \dots, a_\ell, b_2}, \dots, 
\theta_{a_1, \dots, a_\ell, b_\ell}$ 
are linearly independent over $S$. 
\end{lemma}
\begin{proof}
It suffices to show that the determinant of the matrix 
$M=(\theta_{a_1, \dots, a_\ell, b_i}(x_j))_{i, j}$ is nonzero. 
We consider the 
lexicographic order among monomials. Namely, 
$x_1^{p_1}x_2^{p_2}\dots x_\ell^{p_\ell}>x_1^{q_1}x_2^{q_2}\dots 
x_\ell^{q_\ell}$ if and only if $p_1=q_1, q_2=q_2, \dots, p_i=q_i$ and 
$p_{i+1}>q_{i+1}$ for some $i$. The nonzero leading monomial of 
each entry of the Saito matrix $M$ is, up to a constant factor, 
\[
\begin{pmatrix}
x_1^{a_1+\dots+a_\ell+b_1+1}&x_1^{a_1}x_2^{a_2+\dots+a_\ell+b_1+1}&\cdots&x_1^{a_1}\dots x_{\ell-1}^{a_{\ell-1}}x_\ell^{a_{\ell}+b_1+1}\\
x_1^{a_1+\dots+a_\ell+b_2+1}&x_1^{a_1}x_2^{a_2+\dots+a_\ell+b_2+1}&\cdots&x_1^{a_1}\dots x_{\ell-1}^{a_{\ell-1}}x_\ell^{a_{\ell}+b_2+1}\\
\vdots&\vdots&\ddots&\vdots\\
x_1^{a_1+\dots+a_\ell+b_\ell+1}&x_1^{a_1}x_2^{a_2+\dots+a_\ell+b_\ell+1}&\cdots&x_1^{a_1}\dots x_{\ell-1}^{a_{\ell-1}}x_\ell^{a_{\ell}+b_\ell+1}
\end{pmatrix}. 
\]
The product of anti-diagonal entries attains the unique maximum 
monomial 
\[
{
x_1^{\ell a_1+a_2+\dots+a_\ell+b_\ell+1}
x_2^{(\ell-1)a_2+a_3+\dots+a_\ell+b_{\ell-1}+1}
\dots 
x_\ell^{a_\ell+b_1+1}. 
}
\]
Hence $\det M\neq 0$. 
\end{proof}
We can prove the following.

\begin{proposition}
\label{prop:pq}
Let $(\A, m)$ be the multiarrangement defined by $Q(\A, m)$ given by 
\eqref{mbraidl}. 
Then the vector fields 
$\theta_{a_1, \dots, a_\ell, b}$, $\theta_{a_1, \dots, a_\ell, b+1}$, $\dots,$ 
$ \theta_{a_1, \dots, a_\ell, b+\ell-1}$ give a free basis for $(\A, m)$. 
\end{proposition}
\begin{proof}
By Proposition \ref{prop:higherdimmulti} and Lemma \ref{lem:indep}, 
these vector  fields 
are contained in $D(\A, m)$ and are linearly independent 
over $S$. 
Since the sum of the degrees of these vector fields is equal to 
$\deg Q(\A, m)=\ell(\sum_{i=1}^\ell a_i + b)+\frac{\ell(\ell+1)}{2}$, 
by Proposition \ref{prop:szc}, they form a basis of $D(\A, m)$. 
\end{proof}

As another consequence we get a free basis of a multi-braid arrangement 
whose freeness is known from \cite{ann}.

\begin{proposition}
\label{multibraid}
Let $(\mathcal{A},m)$ be the multiarrangement defined by
    \begin{equation*}   
Q(\mathcal{A},m)=   \prod_{1 \le i<j \le \ell} (x_i-x_j)^{a_i+a_j+1}. 
\end{equation*}
Then the vector fields $\theta_{a_1, \dots, a_\ell, b}$ for $b=0, \ldots, \ell-2$ together with $E=\sum_{i=1}^\ell \partial_i$ give a free basis for $(\mathcal{A},m)$.
\end{proposition}
\begin{proof}
By Proposition \ref{prop:higherdimmulti} the vector fields belong to $D(\mathcal{A},m)$. Similarly to Lemma \ref{lem:indep} they are linearly independent over $S$. We also have $\deg E=0$ and 
$$
\sum_{b=0}^{\ell - 2} \deg \theta_{a_1, \dots, a_\ell, b} 
= \frac{\ell(\ell-1)}{2}+(\ell-1)\sum_{i=1}^\ell a_i =\deg Q(\mathcal{A},m).
$$
The statement follows by Proposition \ref{prop:szc}.
\end{proof}

\begin{remark}
Multiarrangements of 
Propositions \ref{prop:pq} and \ref{multibraid} are connected directly 
as follows. Consider the ring homomorphism $\varphi$ from 
$S=\C[x_1, \dots, x_\ell]$ to $T=\C[y_1, \dots, y_\ell, y_{\ell+1}]$ 
defined by $\varphi(x_i)=y_i-y_{\ell+1}$ ($i=1, \dots, \ell$). Then 
we have 
\[
\varphi
\left(
\prod_{i=1}^\ell x_i^{a_i+a_{\ell+1}+1}\cdot 
\prod_{1\leq i<j\leq\ell}(x_i-x_j)^{a_i+a_j+1}
\right)=
\prod_{1\leq i<j\leq\ell+1}(y_i-y_j)^{a_i+a_j+1}. 
\]
The map $\varphi$ induces a linear map 
$f:\Spec T\simeq\C^{\ell+1}\longrightarrow \Spec S\simeq\C^\ell$ 
with $\Ker f=\C\cdot (\bm{e}_1+\dots+\bm{e}_{\ell+1})$, where vectors  $\bm{e}_i$ form the standard basis of $\C^{\ell+1}$. 
Thus the multiarrangement of Proposition \ref{multibraid} in 
$\C^{\ell+1}$ with $a_{\ell+1}=b$ is obtained from that of Proposition \ref{prop:pq} 
in $\C^\ell$ by taking direct product with one-dimensional space. 
\end{remark}

\section{Integral expressions for reflection multiarrangements}

\label{sec:main}

\subsection{Free reflection multiarrangements for monomial groups}
\label{subsec:freemultiarr}

Let $\ell\geq 2$ and $r\geq 2$. 
Let $V=\bigoplus_{i=1}^\ell\C \bm{e}_i$ be an $\ell$-dimensional 
$\C$-linear space. Let $C_r:=\langle\zeta_r\rangle=
\{\zeta_r^k\mid k=0, \dots, r-1\}$ 
be the multiplicative cyclic group of order $r$, 
where $\zeta_r=e^{2\pi\sqrt{-1}/r}$ is 
the primitive $r$-th root of $1$. 
Let $[\ell]:=\{1, \dots, \ell\}$. 
Let $\varepsilon:[\ell]\longrightarrow C_r$ be a map and 
$\sigma\in\frak{S}_\ell$ be a permutation of $[\ell]$. 
Then, define the linear map 
$\varphi_{\sigma, \varepsilon}: V\longmapsto V$ as 
\begin{equation}
\varphi_{\sigma, \varepsilon}(\bm{e}_i)=
\varepsilon(i)\cdot \bm{e}_{\sigma(i)}. 
\end{equation}
The linear maps $\varphi_{\sigma, \varepsilon}$ 
($\sigma\in\frak{S}_\ell$ and $\varepsilon:[\ell]\longrightarrow C_r$) 
form a subgroup of 
$GL(V)$, denoted by $G(r, 1, \ell)$ and called 
the full monomial group. 
The order of the group is 
$|G(r, 1, \ell)|=r^\ell \cdot\ell !$. 

Now let $p$ be a divisor of $r$. The monomial group 
$G(r, p, \ell)$ is the collection of 
$\varphi_{\sigma, \varepsilon}$ such that 
$\prod_{i=1}^\ell \varepsilon(i)^{r/p}=1$. In other words, 
the product $\prod_{i=1}^\ell \varepsilon(i)$ is contained in 
$\langle\zeta_r^p\rangle=C_{r/p}$. 
Clearly, $|G(r, p, \ell)|=r^\ell \cdot\ell !/p$. 

The hyperplane defined by $x_i-\zeta_r^kx_j=0$ 
($i\neq j, k=0, \dots, r-1$) is a reflecting hyperplane of 
$G(r, p,\ell)$. Furthermore, if $r>p$, 
the coordinate hyperplane $x_i=0$ is also a reflecting hyperplane. 
The defining equation $Q(\A)$ of the arrangement $\A$ of 
reflecting hyperplanes is as follows. 
\begin{equation}
Q(\A)=
\begin{cases}
x_1x_2\dots x_\ell\cdot\prod\limits_{1\leq i<j\leq\ell}(x_i^r-x_j^r), 
&\mbox{ if }r>p, \\
\prod\limits_{1\leq i<j\leq\ell}(x_i^r-x_j^r), &\mbox{ if }r=p. 
\end{cases}
\end{equation}

Let 
$u\geq 1$, and $m_i\geq 0$ 
($i=1, \dots, \ell$). 
We are considering the free multiarrangement $(\A, \mu)$ 
of the form 
\begin{equation}
Q(\A, \mu)=\prod_{i=1}^\ell x_i^{m_i}\cdot
\prod_{1\leq i<j\leq\ell}(x_i^r-x_j^r)^u. 
\end{equation}
Hoge, Mano, R\"ohrle, and Stump \cite{hmrs} proved that 
the following multiarrangements are free. 

\begin{theorem}(\cite[Theorem 4.1]{hmrs}) 
\label{thm:hmrs}
Let $m\in\Z_{\geq 0}$ and $k\in\Z_{\geq -m-1}$ 
and $\overline{m}_i$ $(i=1, \dots,, \ell)$ be integers such that 
$0\leq \overline{m}_i\leq r-1$ and 
\[
m_i:=r(m+k)+1+\overline{m}_i\geq 0, 
\]
for $i=1, \dots, \ell$. 
Note that the quotient $q:=\left\lfloor\frac{m_i-1}{r}\right\rfloor=m+k$ 
does not depend on $i$. Set $a:=(\ell-1)r$, $m':=\sum_{i=1}^\ell m_i$, 
and $c:=ma+qr+1$. 

\begin{itemize}
\item[(1)] (The case $u=2m+1$ is odd.) 
Let $(\A, \mu)$ be the multiarrangement defined by 
\begin{equation}
\label{Qmu}
Q(\A, \mu)=\prod_{i=1}^\ell x_i^{m_i}\cdot
\prod_{1\leq i<j\leq\ell}(x_i^r-x_j^r)^{2m+1}. 
\end{equation}
Then $(\A, \mu)$ is free with exponents 
\begin{equation}
\begin{split}
\exp(\A, \mu)=&
\big(c+m'-\ell(qr+1),c+r,c+2r,\dots,c+(\ell-1)r\big) \\
=&\big(mr\ell+rk+1+\sum \overline{m_i},\\
  &r(k+1)+mr\ell+1,r(k+2)+mr\ell+1,\dots,r(k+\ell-1)+mr\ell+1\big).
\end{split}
\end{equation}

\item[(2)] (The case $u=2m$ is even.) 
Let $(\A, \mu')$ be the multiarrangement defined by 
\begin{equation}
\label{Qmuprime}
Q(\A, \mu')=\prod_{i=1}^\ell x_i^{m_i}\cdot
\prod_{1\leq i<j\leq\ell}(x_i^r-x_j^r)^{2m}. 
\end{equation}
Then $(\A, \mu')$ is free with exponents 
\begin{equation}
\exp(\A, \mu')=
(ma+m_1, \dots, ma+m_\ell). 
\end{equation}
\end{itemize}
\end{theorem}

\begin{remark}
If $m_i=0$ then $k=-m-1$ and $\overline{m}_i=r-1$. 
If $m_i=r(m+k)+1$ then we have $k\geq -m$ and $\overline{m}_i=0$. 
\end{remark}

\subsection{Integral expression for $(\A, \mu)$}
\label{subsec:mu}

In this section, we provide integral expressions for the 
basis of free multiarrangements in Theorem \ref{thm:hmrs} (1). 
Let $\lambda(t)=\prod_{i=1}^\ell(t^r-x_i^r)$. 
For $m\in\Z_{\geq 0}$ and $u\in\Z$, 
define the vector field $\eta_u^m$ as 
\begin{equation}
\eta_u^m=
\sum_{i=1}^\ell
\left(
\int^{x_i} t^{ru}\lambda(t)^mdt
\right)\partial_i. 
\end{equation}
Note that $\eta_u^m$ may have poles in general. 
However, it does not have poles for the following cases. 
\begin{proposition}
\label{prop:poly}
Let $\ell, m, k, r, \overline{m}_i$ be as above. 
\begin{itemize}
\item[(1)] 
If $u\geq -m$, then $\eta_u^m$ is a polynomial vector field. 
\item[(2)] 
The vector field 
$\left(\prod_{i=1}^\ell x_i^{\overline{m}_i}\right)\eta_k^m$ 
is a polynomial vector field. 
\end{itemize}
\end{proposition}

\begin{proof}
(1) 
We will see $\eta_u^m x_1=\int^{x_1}t^{ru}\lambda(t)^mdt$ is 
a polynomial. Note that $\lambda(t)$ can be expressed as 
\[
\lambda(t)=A\cdot t^r + B\cdot x_1^r, 
\]
where $A$ and $B$ are polynomials in $t, x_i$. Therefore, 
$t^{ru}\cdot\lambda(t)^m$ has the form 
\[
t^{ru}\cdot\lambda(t)^m=
\sum_{i+j=m}A_i \cdot x_1^{ri}\cdot t^{r(j+u)}, 
\]
where $A_i$ is a polynomial. The degree of $x_1$ in 
$\int^{x_1}t^{ru}\lambda(t)^mdt$ is at least 
$ri+r(j+u)+1= r(m+u)+1\geq 1$ (recall $u\geq -m$). 
Therefore, it is a polynomial. 

(2) 
If $k\geq -m$, the assertion follows from (1). 
Now we consider the case $k=-m-1$. In this case, 
$m_i=0$ and $\overline{m}_i=r-1$ for all $i=1, \dots, \ell$. 
The degree of $x_1$ in 
$\prod_{i=1}^\ell x_i^{\overline{m}_i}\cdot\eta_k^m x_1$ is at least 
\[
\overline{m}_1+r(m+k)+1=0. 
\]
\end{proof}

\begin{theorem}
\label{thm:odd}
Let $\ell, m, k, r, \overline{m}_i$ be as above. 
Then the vector fields 
\begin{equation}
\label{eq:oddbasis}
\left(\prod_{i=1}^\ell x_i^{\overline{m}_i}\right) \eta_k^m, \quad
\eta_{k+1}^m,  \quad \eta_{k+2}^m,  \quad \dots,  \quad  
\eta_{k+\ell-1}^m
\end{equation}
form a basis of $D(\A, \mu)$. 
\end{theorem}

\begin{proof}
By Proposition \ref{prop:poly}, these are polynomial vector fields. 
It is easily seen that 
\[
\eta_u^m x_j=\int^{x_j}t^{ru}\lambda(t)^m dt
\]
is divisible by $x_j^{ru+rm+1}$. When $u\geq k+1$, 
we have 
\[
ru+rm+1\geq r(m+k+1)+1>r(m+k)+\overline{m}_j+1=m_j. 
\]
Similarly, 
$\left(\prod_{i=1}^\ell x_i^{\overline{m}_i}\right) \eta_k^m x_j$ is 
also divisible by $x_j^{\overline{m}_j+rk+rm+1}=x_j^{m_j}$. 

Next we compute the multiplicity of 
$\eta_u^m(x_i-\zeta x_j)$ for $\zeta\in C_r$. 
Since the integrand is invariant under the 
transformation $t=\zeta^{-1}t'$, we have 
\begin{equation}
\begin{split}
\eta_u^m(x_i-\zeta x_j)
&=
\int^{x_i}t^{ru}\lambda(t)^m dt-
\zeta\int^{x_j}t^{ru}\lambda(t)^m dt\\
&=
\int^{x_i}t^{ru}\lambda(t)^m dt-
\zeta\int^{\zeta x_j}(t')^{ru}\lambda(t')^m \zeta^{-1}dt'\\
&=
\int_{\zeta x_j}^{x_i}t^{ru}\lambda(t)^m dt. 
\end{split}
\end{equation}
Since $\lambda(t)^m$ is divisible by $(t-x_i)^m(t-\zeta x_j)^m$, and $\eta_u^m(x_i-\zeta x_j)$ is a polynomial, it follows that $\eta_u^m(x_i-\zeta x_j)$ is divisible by $(x_i-\zeta x_j)^{2m+1}$.

We can prove the linear independence of 
$\eta_k^m, \eta_{k+1}^m, \dots, \eta_{k+\ell-1}^m$ 
in a similar way to Lemma \ref{lem:indep}. 
It is also straightforward that degrees of these vector fields 
are exactly equal to the exponents $\exp(\A, \mu)$ in 
Theorem \ref{thm:hmrs} (1). Hence, by Saito-Ziegler criterion, 
the vector fields (\ref{eq:oddbasis}) 
form a free basis of $D(\A, \mu)$. 
\end{proof}

\subsection{Integral expression for $(\A, \mu')$}
\label{subsec:muprime}

In this section, we provide integral expressions for the 
basis of free multiarrangements in Theorem \ref{thm:hmrs} (2). 
Let $\lambda_i(t)=\frac{\lambda(t)}{t^r-x_i^r}$. 

For $m\in\Z_{>0}$ and $1\leq i\leq\ell$, 
define the vector field $\sigma_i^m$ as 
\begin{equation}
\sigma_i^m=\sum_{j=1}^\ell
\left(
\int^{x_j}
t^{r(k+1)}\lambda(t)^{m-1}\lambda_i(t)dt
\right)
\partial_j. 
\end{equation}

\begin{theorem}
\label{thm:even}
Let $\ell, m, k, r, \overline{m}_i$ be as above. 
Then the vector fields 
\begin{equation}
\label{eq:evenbasis}
x_1^{\overline{m}_1}\sigma_1^m,\ 
x_2^{\overline{m}_2}\sigma_2^m,\ \dots, \  
x_\ell^{\overline{m}_\ell}\sigma_\ell^m
\end{equation}
form a basis of $D(\A, \mu')$. 
\end{theorem}

\begin{proof}
First we prove that $x_i^{\overline{m}_i}\sigma_i^m$ is a polynomial 
vector field and $x_i^{\overline{m}_i}\sigma_i^m x_j$ is divisible by 
$x_j^{m_j}$. 
The exponent of $x_n$ in each monomial of 
\[
x_i^{\overline{m}_i}\sigma_i^m x_j=
x_i^{\overline{m}_i}\int^{x_j}
t^{r(k+1)}\lambda(t)^{m-1}\lambda_i(t)dt
\]
is clearly nonnegative if $n\neq j$. Now we consider the 
exponent of $x_j$. We note that $x_i^{\overline{m}_i}\sigma_i^m x_j$ is divisible 
by 
$x_j^{r(k+1)+mr+1}$ if $i\neq j$,  
and it is divisible by 
$x_j^{\overline{m}_i+r(k+1)+r(m-1)+1}$ if $i= j$. 
In each case, the exponent is at least $m_i\geq 0$. 

Next we can prove that 
$x_i^{\overline{m}_i}\sigma_i^m (x_j-\zeta x_n), (\zeta\in C_r)$ is 
divisible by $(x_j-\zeta x_n)^{2m}$ in a similar way to 
Theorem \ref{thm:odd}. We can also easily check that 
the degrees of these vector fields are 
\[
\deg x_i^{\overline{m}_i}=\overline{m}_i+r(k+1)+\ell r(m-1)+
r(\ell -1)+1=m_i+ma, 
\]
which are equal to the exponents in Theorem \ref{thm:hmrs} (2). 

It remains to prove that $\sigma_1^m, \dots, \sigma_\ell^m$ are 
linearly independent over $S$. Let $M=(\sigma_i^m x_j)_{i, j}$. 
To prove that $\det M\neq 0$, we look at the leading term 
of each entry with respect to the lexicographic order. 
To look at the leading term, we need 
\begin{equation}
\label{eq:nonzero}
\int^x t^{r n}(t^r-x^r)^m dt\neq 0, 
\end{equation}
($r\ge 2$, $m\geq 0$ and $n\in\Z$), which is a nonzero multiple of 
$x^{r(n+m)+1}$. Actually, (\ref{eq:nonzero}) is equivalent to 
$\int_1^{\zeta_r} t^{r n}(t^r-1)^m dt\neq 0$, which can be proved 
by using integration by parts and the induction on $m$. 
Then the nonzero leading monomials of entries of $M$ are (up to a constant factors) as follows: 
\[
{\footnotesize 
\begin{pmatrix}
x_1^{rk+\ell rm +1}&
x_1^{r(m-1)}x_2^{r(k+1)+rm(\ell-1)+1}&
x_1^{r(m-1)}x_2^{rm}x_3^{r(k+1)+rm(\ell-2)+1}&
\cdots\\
x_1^{rk+\ell rm +1}&
x_1^{rm}x_2^{rk+rm(\ell-1)+1}&
x_1^{rm}x_2^{r(m-1)}x_3^{r(k+1)+rm(\ell-2)+1}&
\cdots\\
x_1^{rk+\ell rm +1}&x_1^{rm}x_2^{rk+rm(\ell-1)+1}&
x_1^{rm}x_2^{rm}x_3^{rk+rm(\ell-2)+1}&
\cdots\\
x_1^{rk+\ell rm +1}&x_1^{rm}x_2^{rk+rm(\ell-1)+1}&
x_1^{rm}x_2^{rm}x_3^{rk+rm(\ell-2)+1}&
\cdots\\
\vdots&\vdots&\vdots&\ddots
\end{pmatrix}.
}
\]
In $j$-th column, the monomial increases strictly up to 
$j$-th place, and 
it is equal to $x_1^{rm}\ldots x_{j-1}^{rm} x_j^{r k + r m(\ell-j+1)+1}$ at the $j$-th place and below. 
In the 
determinant of $M$, the product of diagonal entries 
produces the unique nonzero leading monomial of 
$\det M$. Therefore $\det M\neq 0$. 
\end{proof}

\subsection{Special cases}
\label{subsec:special}
 
Let us list a few examples illustrating Theorems \ref{thm:odd} and \ref{thm:even}. 
\begin{example}
\label{ex:catalan}
Consider the case of Theorem \ref{thm:odd} when $k=0$ and $\overline{m}_i=0$. 
Then 
the defining equation of the multiarrangement is 
\[
\prod_{i=1}^\ell x_i^{rm+1}\cdot
\prod_{1\leq i<j\leq\ell}(x_i^r-x_j^r)^{2m+1}.
\] 
The vector fields 
\begin{equation}
\eta_u^m=
\sum_{i=1}^\ell
\left(
\int^{x_i} t^{ru}\prod_{j=1}^\ell(t^r-x_j^r)^m
dt
\right)\partial_i, 
\end{equation}
$0\leq u\leq \ell-1$, form a basis. 
\end{example}

\begin{example}
\label{ex:Dcatalan}
Consider the case of Theorem \ref{thm:odd} when 
$k=-m-1$ and $\overline{m}_i=r-1$. 
Then 
the defining equation of the multiarrangement is 
\[
\prod_{1\leq i<j\leq\ell}(x_i^r-x_j^r)^{2m+1}.
\] 
The vector fields 
\begin{equation*}
(x_1\dots x_\ell)^{r-1}\eta_{-m-1}^m =
(x_1\dots x_\ell)^{r-1}
\sum_{i=1}^\ell
\left(
\int^{x_i} t^{-r(m+1)}\prod_{j=1}^\ell(t^r-x_j^r)^m
dt
\right)\partial_i, 
\end{equation*}
and
\begin{equation*}
\eta_u^m =
\sum_{i=1}^\ell
\left(
\int^{x_i} t^{ru}\prod_{j=1}^\ell(t^r-x_j^r)^m
dt
\right)\partial_i, 
\qquad (-m\leq u\leq -m+\ell-2), 
\end{equation*}
form a basis. 
\end{example}

\begin{example}
\label{ex:typeB}
(Coxeter arrangement of type $B_\ell$) 
Let $r=2$ and $m\geq 0$. 

(1) Let $m_i=2m+1$, hence, $k=0$, $\overline{m}_i=0$ 
($i=1, \dots, \ell$). Then the defining equation 
\eqref{Qmu} takes the form 
\[
\prod_{i=1}^\ell x_i^{2m+1}\cdot
\prod_{1\leq i<j\leq\ell}(x_i^2-x_j^2)^{2m+1}. 
\]
The vector fields $\eta_0^m, \eta_1^m, \dots, \eta_{\ell-1}^m$, 
or, 
more explicitly, 
\[
\sum_{i=1}^\ell\left(
\int_0^{x_i}t^{2u}
\prod_{j=1}^\ell(t^2-x_j^2)^m\right)\partial_i, 
\]
($u=0, 1, \dots, \ell-1$) form a basis of $D(B_\ell, 2m+1)$. 

(2) Let $m_i=2m$, hence, $k=-1$, $\overline{m}_i=1$ 
($i=1, \dots, \ell$). 
Then the defining equation 
\eqref{Qmuprime} takes the form 
\[
\prod_{i=1}^\ell x_i^{2m}\cdot
\prod_{1\leq i<j\leq\ell}(x_i^2-x_j^2)^{2m}. 
\]
The vector fields 
$x_1\sigma_1^m, x_2\sigma_2^m, \dots, x_\ell\sigma_\ell^m$, 
or, 
more explicitly, 
\[
x_i\cdot\sum_{j=1}^\ell
\left(
\int_0^{x_j}
\frac{1}{t^2-x_i^2}
\prod_{k=1}^\ell(t^2-x_k^2)^m\right)\partial_j, 
\]
($i=1, \dots, \ell$) form a basis of $D(B_\ell, 2m)$. 
\end{example}

\begin{example}
\label{ex:typeD}
(Coxeter arrangement of type $D_\ell$) 
Let $r=2$ and $m\geq 0$. 

(1) Let $m_i=0$, hence, $k=-m-1$, $\overline{m}_i=1$ 
($i=1, \dots, \ell$). Then by Theorem \ref{thm:odd}, 
the vector fields 
\[
\begin{split}
\prod_{i=1}^\ell x_i\cdot&\sum_{j=1}^\ell\left(
\int^{x_j}t^{-2(m+1)}\prod_{k=1}^\ell(t^2-x_k^2)^m dt\right)\partial_j, \\
&
\sum_{j=1}^\ell\left(
\int^{x_j}t^{-2u}\prod_{k=1}^\ell(t^2-x_k^2)^m dt\right)\partial_j, 
\end{split}
\]
($u=m, m-1, \dots, m-\ell+2$) form a basis of the multiarrangement 
defined by $\prod_{1\leq i<j\leq\ell}(x_i^2-x_j^2)^{2m+1}$. 

(2) 
Again let $m_i=0$, hence, $k=-m-1$, $\overline{m}_i=1$ 
($i=1, \dots, \ell$). Then by Theorem \ref{thm:even}, 
the vector fields 
\[
x_i\cdot\sum_{j=1}^\ell
\left(
\int^{x_j}\frac{t^{-2m}}{t^2-x_i^2}\prod_{k=1}^\ell (t^2-x_k^2)^m dt
\right)
\partial_j
\]
($i=1, \dots, \ell$), form a basis of the multiarrangement 
defined by $\prod_{1\leq i<j\leq\ell}(x_i^2-x_j^2)^{2m}$. 
\end{example}

\begin{remark}
By the results of \cite{aefy}, Examples \ref{ex:catalan} -- \ref{ex:typeD} and, more generally, Theorems \ref{thm:odd}--\ref{thm:even} also lead to integral formulas of certain quasi-invariants associated with the monomial groups $G(r,p,\ell)$ with $p=1, r$.
\end{remark}

\section{Primitive derivation for well-generated monomial groups}

\label{sec:prim}

\subsection{The invariant rings of monomial groups}

\label{subsec:inv}

Let 
\[
e_i(x_1, \dots, x_\ell)=\sum_{1\leq s_1<\dots <s_i\leq\ell}
x_{s_1}\dots x_{s_i}
\]
be the elementary symmetric polynomial of degree $i$. The basic 
invariants, which are the generators of the invariant ring 
$\C[x_1, \dots, x_\ell]^{G(r, p, \ell)}$, can be explicitly 
described using $e_i$. Namely, let 
\[
\begin{cases}
P_i&:=(-1)^i e_i(x_1^r, \dots, x_\ell^r), \mbox{ for } 1\leq i\leq\ell-1\\
P_\ell&:=(x_1x_2\dots x_\ell)^{r/p}. 
\end{cases}
\]
Note that $P_\ell^p=e_\ell(x_1^r, \dots, x_\ell^r)$. 
In particular, if $p=1$, we have $P_\ell=e_\ell(x_1^r, \dots, x_\ell^r)$. 
We also have 
\begin{equation}
\label{eq:coeff}
\lambda(t)=t^{r\ell}+P_1(x) t^{r(\ell-1)}+\dots+
P_{\ell-1}(x) t^r +(-1)^\ell P_\ell(x)^p. 
\end{equation}
We formally have 
\[
\frac{\partial \lambda(t)}{\partial P_i}\dot{=}
\begin{cases}
t^{r(\ell-i)}, \mbox{ if } 1\leq i\leq \ell-1\\
P_\ell(x)^{p-1}, \mbox{ if } i=\ell. 
\end{cases}
\]
Recall that (\cite{bro}) 
$P_1, \dots, P_\ell$ are algebraically independent, and 
\[
\C[x_1, \dots, x_\ell]^{G(r, p, \ell)}=\C[P_1, \dots, P_\ell]. 
\]
The Jacobian of the basic invariants is 
\begin{equation}
\label{eq:jacobian}
\det\left(\frac{\partial P_i}{\partial x_j}\right)\dot{=}\ 
(x_1 x_2\dots x_\ell)^{\frac{r}{p}-1}\prod_{1\leq i<j\leq\ell}(x_i^r-x_j^r). 
\end{equation}
Later, the basic invariant of the highest degree plays an important role.

\subsection{Order of reflections} 

\begin{definition}
(1) Denote by $\bm{\beta}$ the multiplicity assigning $1$ to each 
coordinate hyperplane $\{x_i=0\}$, 
and assigning 0 to the hyperplanes $\{x_i-\zeta x_j=0\}$, where $\zeta\in C_r$. 

(2) Denote by $\bm{\delta}_r$ the multiplicity assigning $1$ to each 
hyperplane of the form $\{x_i-\zeta x_j=0\}$, where $\zeta\in C_r$, 
and assigning $0$ to each 
coordinate hyperplane $\{x_i=0\}$. 

(3) For the group 
$G(r, p, \ell)$, denote by $\bm{\omega}$ 
the multiplicity assigning the order of the reflection to each 
reflecting hyperplane. More explicitly, $\bm{\omega}$  is 
expressed as follows. 
\[
\bm{\omega}=
\begin{cases}
\frac{r}{p}\bm{\beta}+2\bm{\delta}_r, \mbox{ if } p<r,\\
\\
2\bm{\delta}_r, \mbox{ if } p=r. 
\end{cases}
\]
\end{definition}

\subsection{Primitive derivation $D$}

Let $G=G(r, 1, \ell)$ or $G(r, r, \ell)$. 
Denote the invariant ring by $R:=S^G$. 
Among the vector fields 
$\frac{\partial}{\partial P_1}, \dots, 
\frac{\partial}{\partial P_\ell}$, 
there exists unique one with the 
lowest degree, which we denote by $D$, namely, 
\[
D=
\begin{cases}
\frac{\partial}{\partial P_\ell}, \mbox{ if } p=1,\\
\frac{\partial}{\partial P_{\ell-1}}, \mbox{ if } p=r.
\end{cases}
\]
This vector field is canonically determined up to nonzero 
scalar multiplication, it 
is called the primitive derivation.

\subsection{The action of the primitive derivation}

We will use the notation from Example \ref{ex:catalan}. 
In this section, 
we describe the action of primitive derivation on the 
module $D(\A, m\cdot\bm{\omega}+1)^G$.

Denote by $\nabla$ the integrable connection with flat sections 
$\partial_1, \dots, \partial_\ell$. More explicitly, for polynomial 
vector fields $\delta$ and $\eta=\sum_{i=1}^\ell f_i\partial_i$, we define 
\[
\nabla_\delta\eta=\sum_{i=1}^\ell (\delta f_i)\partial_i. 
\]
The vector fields $\frac{\partial}{\partial P_i}$ can be considered as  
rational vector fields on $V$. For example, 
\begin{equation}
\label{eq:prim}
\frac{\partial}{\partial P_\ell}=
\frac{1}{Q}\det
\begin{pmatrix}
\frac{\partial P_1}{\partial x_1}& 
\frac{\partial P_2}{\partial x_1}& 
\dots &
\frac{\partial P_{\ell-1}}{\partial x_1}& 
\frac{\partial}{\partial x_1}
\\
\frac{\partial P_1}{\partial x_2}& 
\frac{\partial P_2}{\partial x_2}& 
\dots &
\frac{\partial P_{\ell-1}}{\partial x_2}& 
\frac{\partial}{\partial x_2}
\\
\vdots&\vdots&\ddots&\vdots&\vdots
\\
\frac{\partial P_1}{\partial x_\ell}& 
\frac{\partial P_2}{\partial x_\ell}& 
\dots &
\frac{\partial P_{\ell-1}}{\partial x_\ell}& 
\frac{\partial}{\partial x_\ell}
\end{pmatrix}, 
\end{equation}
where $Q=\det(\frac{\partial P_i}{\partial x_j})$ is the 
Jacobian as in (\ref{eq:jacobian}). 

The primitive derivation $\nabla_D$ acts as follows. 
\begin{proposition}
(\cite{hmrs}) 
Let $G$ be a well-generated complex reflection group. Let $T$ be 
the subring of $R=S^G$ defned as 
\[
T=\{P\in S^G\mid DP=0\}. 
\]
The primitive derivation $\nabla_D$ induces an isomorphism of $T$-modules 
\begin{equation}
\nabla_D: D(\A, (m+1)\bm{\omega}+1)^G\stackrel{\simeq}{\longrightarrow} 
D(\A, m\bm{\omega}+1)^G, 
\end{equation}
for $m\geq 0$. It is free $S^G$-module and the basis also spans 
$D(\A, m\bm{\omega}+1)$. 
\end{proposition}
Therefore, the basis of $D(\A, m\bm{\omega}+1)$ can be described by 
using the inverse primitive derivation 
$\nabla_D^{-1}$, which is difficult to describe in general. However, 
for well-generated 
groups $G(r, 1, \ell)$ and $G(r, r, \ell)$, we can describe 
$\nabla_D^{-1}$ in terms of the integral expression.

\subsection{The case $G=G(r, 1, \ell)$}

In this case, $D=\frac{\partial}{\partial P_\ell}$. 
As in Example \ref{ex:catalan}, the vector fields 
\begin{equation}
\label{eq:invbasis}
\eta_0^m, \eta_1^m, \dots, \eta_{\ell-1}^m 
\end{equation}
form a basis of 
$D(\A, m\bm{\omega}+1)$. 
Furthermore, since $\eta_u^m$ is $G$-invariant, 
the basis (\ref{eq:invbasis}) forms the $S^G$-basis of 
$D(\A, m\bm{\omega}+1)^G$. 
By the expression (\ref{eq:coeff}), we have 
\[
D\lambda(t)^m=(-1)^\ell m\lambda(t)^{m-1}. 
\]
Thus we have (see also \cite[Remark 2.5]{suy-yos}) 
\[
\nabla_D^{-m}\eta_k^0\dot{=}\eta_k^m. 
\]

\subsection{The case $G=G(r, r, \ell)$}

In this case, $D=\frac{\partial}{\partial P_{\ell-1}}$. As in Example \ref{ex:Dcatalan}, 
\begin{equation}
\label{eq:Dinvbasis}
P_\ell^{r-1}\eta_{-m-1}^m, 
\eta_{-m}^m, \eta_{-m+1}^m, \dots, \eta_{-m+\ell-2}^m 
\end{equation}
form a basis of $D(\A, m\bm{\omega}+1)$. 
Furthermore, since $\eta_u^m$ is $G$-invariant, 
the basis (\ref{eq:Dinvbasis}) forms the $S^G$-basis of 
$D(\A, m\bm{\omega}+1)^G$. 
By the expression (\ref{eq:coeff}), we have 
\[
D\lambda(t)^m=m t^r \lambda(t)^{m-1}, 
\]
and we also have 
\[
\nabla_D\eta_u^m 
\dot{=} 
\eta_{u+1}^{m-1}.
\]
Thus we have 
\[
\nabla_D^{-m}P_\ell^{r-1}\eta_{-1}^0\dot{=}P_\ell^{r-1}\eta_{-m-1}^m, \ \ 
\nabla_D^{-m}\eta_{u}^0\dot{=}
\eta_{-m+u}^m, (0\leq u\leq\ell-2). 
\]

\section{A conjecture on $B_2$-Catalan arrangement}

\label{sec:conj}

In this section, we propose a conjecture regarding the basis of 
the Catalan arrangement of type $B_2$. 

\begin{definition}
Let $\ell>0$ and $m\geq 0$. 
Define the extended Catalan arrangement $\Cat(B_\ell, m)$ 
of type $B_\ell$ by 
\[
\prod_{i=1}^\ell\prod_{k=-m}^m(x_i-k)\cdot
\prod_{1\leq i<j\leq\ell}\prod_{k=-m}^m
(x_i+x_j-k)(x_i-x_j-k). 
\]
\end{definition}
Note that the extended Catalan arrangement determines 
a central arrangement in $\R^{\ell+1}$, 
the so-called coning \cite{ot} $c\Cat(B_\ell, m)$. By taking 
the Ziegler restriction to the hyperplane at infinity, 
we obtain the multiarrangement defined by 
\[
\prod_{i=1}^\ell x_i^{2m+1}
\cdot
\prod_{1\leq i<j\leq\ell}
(x_i^2-x_j^2)^{2m+1}. 
\]

\begin{figure}[htbp]
\centering
\begin{tikzpicture}[scale=0.8]


\draw[thick, black] (2.0,0) -- +(0, 6); 
\draw[thick, black] (2.5,0) -- +(0, 6); 
\draw[thick, black] (3,0) -- +(0, 6); 
\draw[thick, black] (3.5,0) -- +(0, 6); 
\draw[thick, black] (4,0) -- +(0, 6); 

\draw[thick, black] (0,2.0) -- +(6, 0); 
\draw[thick, black] (0,2.5) -- +(6, 0); 
\draw[thick, black] (0,3) -- +(6, 0); 
\draw[thick, black] (0,3.5) -- +(6, 0); 
\draw[thick, black] (0,4) -- +(6, 0); 

\draw[thick, black] (0,1) -- +(5, 5); 
\draw[thick, black] (0,0.5) -- +(5.5, 5.5); 
\draw[thick, black] (0,0) -- +(6, 6); 
\draw[thick, black] (0.5,0) -- +(5.5, 5.5); 
\draw[thick, black] (1,0) -- +(5, 5); 

\draw[thick, black] (0,5) -- +(5, -5); 
\draw[thick, black] (0,5.5) -- +(5.5, -5.5); 
\draw[thick, black] (0,6) -- +(6, -6); 
\draw[thick, black] (0.5,6) -- +(5.5, -5.5); 
\draw[thick, black] (1,6) -- +(5, -5); 

\draw[thick, black] (7,3) -- +(4,0) node[right] {$5$}; 
\draw[thick, black] (9,1) -- +(0,4) node[above] {$5$}; 
\draw[thick, black] (7,1) -- +(4,4) node[right] {$5$}; 
\draw[thick, black] (7,5) -- +(4,-4) node[right] {$5$}; 


\end{tikzpicture}
\caption{$\Cat(B_2, 2)$ and associated multiarrangement} 
\label{fig:B2cat}
\end{figure}
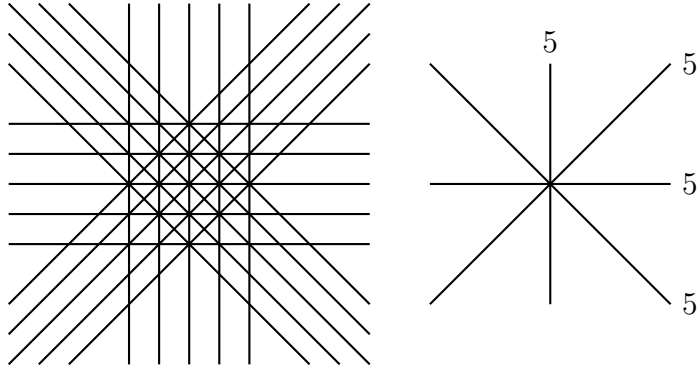
The coning $c\Cat(B_\ell, m)$ is known to be free \cite{yos-char}. 
However, the explicit basis has not been constructed even for 
$B_2$ case. 

Let $m, i\geq 0$. Define the polynomial $f^m_i(x, y)$ to be 
\[
f^m_i(x, y):=
\int_0^x t^{2i}(t^2-x^2)^m(t^2-y^2)^mdt. 
\]

\begin{example}
We have 
\[ 
\begin{split}
f^1_0(x,y)&\dot{=}x^5-5x^3y^2,\\
f^1_1(x,y)&\dot{=}3x^7-7x^5y^2,\\
f^1_2(x,y)&\dot{=}5x^9-9x^7y^2,\\
f^1_3(x,y)&\dot{=}7x^{11}-11x^9y^2,\\
f^2_0(x,y)&\dot{=}x^9-6x^7y^2+21x^5y^4,\\
f^2_1(x,y)&\dot{=}5x^{11}-22x^9y^2+33x^7y^4,\\
f^2_2(x,y)&\dot{=}35x^{13}-130x^{11}y^2+143x^9y^4,\\
f^2_3(x,y)&\dot{=}21x^{15}-70x^{13}y^2+65x^{11}y^4,\\
f^3_0(x, y)&\dot{=}5x^{13}-39x^{11}y^2+143x^9y^4-429x^7y^6,\\
f^3_1(x, y)&\dot{=}7x^{15}-45x^{13}y^2+117x^{11}y^4-143x^9y^6,\\
f^3_2(x, y)&\dot{=}21x^{17}-119x^{15}y^2+255x^{13}y^4-221x^{11}y^6,\\
f^3_3(x, y)&\dot{=}231x^{19}-1197x^{17}y^2+2261x^{15}y^4-1615x^{13}y^6.
\end{split}
\]
\end{example}
The polynomial $f^m_i(x, y)$ has the following form, 
\begin{equation}
\label{eq:fmi}
f^m_i(x, y)=\sum_{k=0}^m c_kx^{4m+2i+1-2k}y^{2k}. 
\end{equation}
We already know that the vector field 
$f^m_i(x, y)\partial_x+f^m_i(y, x)\partial_y$ is contained in 
$D(B_2, 2m+1)$, which is equivalent to 
\[
f^m_i(x, y)\in (x^{2m+1}), \mbox{ and }
f^m_i(x, y)\pm f^m_i(y, x)\in ((x\pm y)^{2m+1}). 
\]

In \cite{suy-yos}, for type $A_\ell$, 
the basis of $D(c\Cat(\A_\ell, m))$ was 
constructed using the discrete analogue of the integral expression. 
However, for type $B_\ell$, the discrete integral does not work. 
We need to define a suitable deformation of the integral expression. 
In order to describe the deformation, we 
require deformed power of functions. 
For a function of one variable $f(x)$, 
define the falling power $f(x)^{\underline{n}}$ and the 
rising power $f(x)^{\overline{n}}$ as 
\[
\begin{split}
f(x)^{\underline{n}}&=f(x)f(x-1)\cdots f(x-n+1),\\
f(x)^{\overline{n}}&=f(x)f(x+1)\cdots f(x+n-1). 
\end{split}
\]
(We may also set $f(x)^{\underline{0}}=f(x)^{\overline{0}}=1$.) 

\begin{definition}
For a homogeneous polynomial of the form 
\[
f(x, y)=
\sum_{k=0}^m c_k x^{2p+1-2k}y^{2k}, 
\]
where $p\geq m\geq 0$ and $c_k\in\C$, define the deformation 
$\widetilde{f}(x, y)$ as 
\[
\widetilde{f}(x, y)=
\sum_{k=0}^m c_k (x+p-k)^{\underline{2p+1-2k}}(y+p-m)^{\underline{k}}(y-p+m)^{\overline{k}}. 
\]
\end{definition}

Let $g(x, y)\in\C[x, y]$ and $\eta=g(x, y)\partial_x+g(y,x)\partial_y$. 
The homogenization (\cite[\S 4.1]{suy-yos}) of $\eta$ is contained 
in $D(c\Cat(B_2, m))$ 
(which is also equivalent to  $\eta\in D(\Cat(B_2, m))$)
if and only if the following conditions 
are satisfied: 
\begin{equation}
\label{eq:divisibl}
\begin{split}
&
 \eta x=
g(x, y) \mbox{ is divisible by } \prod_{k=-m}^m(x-k), \mbox{ and }\\
&
\eta (x\pm y)=
g(x, y)\pm g(y, x) \mbox{ is divisible by }\prod_{k=-m}^m(x\pm y-k). 
\end{split}
\end{equation}

\begin{example}
The following functions $\widetilde{f^m_i}(x, y)$ 
satisfy the conditions (\ref{eq:divisibl}): 
\[
\begin{split}
\widetilde{f^1_0}(x,y)
\dot{=}&(x+2)^{\underline{5}}-5(x+1)^{\underline{3}}(y+1)^{\underline{1}}(y-1)^{\overline{1}}\\
=&x(x^2-1)(x^2-4)-5x(x^2-1)(y^2-1),\\
\widetilde{f^1_1}(x,y)\dot{=}&3(x+3)^{\underline{7}}-7(x+2)^{\underline{5}}(y+2)^{\underline{1}}(y-2)^{\overline{1}},\\
&\cdots\\
\widetilde{f^3_3}(x, y)\dot{=}&231(x+9)^{\underline{19}}-1197(x+8)^{\underline{17}}(y+6)^{\underline{1}}(y-6)^{\overline{1}}\\
&+2261(x+7)^{\underline{15}}(y+6)^{\underline{2}}(y-6)^{\overline{2}}-1615(x+6)^{\underline{13}}(y+6)^{\underline{3}}(y-6)^{\overline{3}}.
\end{split}
\]
\end{example}
\begin{conjecture}
\label{conj:divisibl}
For any $m, i\geq 0$, $\widetilde{f^m_i}(x, y)+\widetilde{f^m_i}(y, x)$ 
is divisible by 
$\prod_{k=-m}^m(x+y-k)$. 
\end{conjecture}
Taking into account the fact that 
$\widetilde{f^m_i}(x, y)$ is divisible by 
$\prod_{k=-m}^m(x-k)$ (which is straightforward from (\ref{eq:fmi})), 
Conjecture \ref{conj:divisibl} implies that vector fields 
\[
\begin{split}
\widetilde{\eta^m_0}
&=
\widetilde{f^m_0}(x, y)\partial_x+\widetilde{f^m_0}(y, x)\partial_y,\\
\widetilde{\eta^m_1} 
&=
\widetilde{f^m_1}(x, y)\partial_x+\widetilde{f^m_1}(y, x)\partial_y
\end{split}
\]
form a basis of $D(\Cat(B_2, m))$, and that 
the 
homogenization of these vector fields 
together with the Euler vector field form a basis of 
$D(c\Cat(B_2, m))$.

\end{document}